\documentclass{amsart}
\usepackage{amssymb}
\usepackage{bbm}
\usepackage{hyperref}

\newcommand{\R}{\mathbbm R}

\newcommand{\N}{\mathbbm N}

\newcommand{\Hone}{H^1(M;g_0)}

\newtheorem{theorem}{Theorem}
\newtheorem{proposition}[theorem]{Proposition}
\newtheorem{lemma}[theorem]{Lemma}
\newtheorem{corollary}[theorem]{Corollary}
\newtheorem{observation}[theorem]{Observation}

\theoremstyle{definition}

\theoremstyle{remark}
\newtheorem{remark}[theorem]{Remark}

\numberwithin{equation}{section}
\numberwithin{theorem}{section}

\DeclareMathOperator{\vol}{Vol}

\newcommand{\coloneqq}{\,\raise0.08ex\hbox{\textnormal{:}}\!\!=}

\def\XXint#1#2#3{{\setbox0=\hbox{$#1{#2#3}{\int}$}
     \vcenter{\hbox{$#2#3$}}\kern-.5\wd0}}

\begin{document}

\title[Compactness issues and bubbling phenomena]
{Compactness issues and bubbling phenomena for the prescribed Gaussian curvature equation on the torus}

\author{Luca Galimberti}
\address[Luca Galimberti]{Departement Mathematik\\ETH-Z\"urich\\CH-8092 Z\"urich}
\email{luca.galimberti@math.ethz.ch}
\thanks{Supported by SNF grant 200021\_140467 / 1. }
\date{\today}

\begin{abstract} 
In the spirit of the previous paper \cite{Borer-Galimberti-Struwe}, where we dealt with the case of a 
closed Riemann surface $(M,g_0)$ of genus greater than one, here we study the behaviour of the conformal metrics
$g_\lambda$ of prescribed Gauss curvature $K_{g_\lambda} = f_0 + \lambda$ on the torus, when the parameter $\lambda$ tends to one of the boundary points of the interval of existence of $g_\lambda$, and we characterize their ``bubbling behavior'' as in \cite{Borer-Galimberti-Struwe}.
\end{abstract} 

\maketitle

\section{Introduction} 

Consider a closed, connected Riemann surface $M$, whose Euler characteristic $\chi(M)$ is zero, 
endowed with a smooth background metric $g_0$. In view of the uniformization theorem, it is possible to
assume that the Gauss curvature $K_{g_0}$ of $g_0$ vanishes identically. 

The prescribed Gauss curvature equation, which links the curvature of $g_0$ to the curvature $K_g$ of a conformal metric $g=e^{2u}g_0$, then reads as
\[
K_g = -e^{-2u} \Delta_{g_0}u 
\]
Moreover, for convenience, we normalize the volume of $(M,g_0)$ to unity.

Consider a smooth non-constant function $f_0:M\to \R$ with $\max_{p\in M}f_0(p)=0$, all of whose maximum points are non-degenerate, and define for $\lambda \in \R$
\[
f_\lambda := f_0+\lambda
\]
A natural question is to understand for which values $\lambda$ the function $f_\lambda$ is the Gauss curvature of a metric conformal to $g_0$. That is equivalent to ask for which values of $\lambda$, the equation
\begin{equation}\label{eqn: the Gauss curvature}
 -\Delta_{g_0}u = f_\lambda e^{2u} 
\end{equation}
admits a solution. The paper \cite{Kazdan-Warner74} completely answers the question, giving 
necessary and sufficient conditions for solving the equation above. More precisely, equation (\ref{eqn: the Gauss curvature}) has a solution if and only if
\[
\int_M f_\lambda d\mu_{g_0} =\int_M f_0 d\mu_{g_0} + \lambda< 0
\]
and $f_\lambda$ is sign changing. (Recall that the volume is equal to one.)
Thus, by taking account of the assumptions made on $f_0$, we find
that equation (\ref{eqn: the Gauss curvature}) is solvable if and only if 
\[
0= -\max_M f_0 < \lambda < -\overline{f_0},
\] 
where $\overline{f_0} := \int_M f_0 d\mu_{g_0}$. Set
\[
\Lambda :=(0,-\overline{f_0}), \;\;\; -\overline{f_0}:= \lambda_{max}     .
\]
Our goal in this paper is to study the behaviour of the set of solutions of (\ref{eqn: the Gauss curvature}) when $\lambda$ approches either 0 or $\lambda_{max}$, a problem left open in \cite{Kazdan-Warner74} and which 
we solve by means of a blow-up analysis in the spirit of \cite{Borer-Galimberti-Struwe}. Our main results are the following:
\begin{theorem}\label{thm: first main result}
Let $f_0\leq 0$ be a smooth, non-constant function, all of whose maximum points $p_0$ 
are non-degenerate with $f_0(p_0)=0$, and for $\lambda\in\R$ let $f_{\lambda}=f_0+\lambda$.

Then there exists a sequence $\lambda_n\downarrow 0$, a sequence 
$u_n$ of solutions of the equation
\[
-\Delta_{g_0}u_n = f_{\lambda_n} e^{2u_n} 
\]
and there exists $I\in\N$ such that, for suitable $p^{(i)}_n\to p^{(i)}_{\infty}\in M$ with $f_0(p^{(i)}_{\infty})=0$, $1\le i \le I$, we obtain $u_n(p_n^{(i)})\to +\infty$ and one of the following:

\textbf{i)} $u_n\to -\infty$ locally uniformly on compact domains of 
$M_{\infty}:=M\setminus\{p^{(i)}_{\infty}; \; 1\leq i \leq I \}$ 

\textbf{ii)} For suitable $r_n^{(i)}\downarrow 0$, the following holds:

a) We have smooth convergence $u_n\to u_{\infty}$ locally on 
$M_{\infty}$ and $u_{\infty}$ induces a complete metric 
$g_{\infty}=e^{2u_{\infty}}g_0$  on $M_{\infty}$ of finite total curvature $K_{g_{\infty}}=f_0$.

b) For each $1\le i \le I$, either 1) there holds 
$r_n^{(i)}/\sqrt{\lambda_n}\to 0$ and in local 
conformal coordinates around $p^{(i)}_n$ we have 
\[w_n(x):=u_n(r^{(i)}_nx)-u_n(0)+ \log 2\to w_{\infty}(x)
=\log\big(\frac2{1+|x|^2}\big)\]
smoothly locally in $\R^2$, where $w_{\infty}$ induces a spherical metric 
$g_{\infty}=e^{2w_{\infty}}g_{\R^2}$ of curvature $K_{g_\infty}=1$ on $\R^2$,
or 2) we have $r_n^{(i)}=\sqrt{\lambda_n}$,
and in local conformal coordinates around $p^{(i)}_{\infty}$ with a constant $c^{(i)}_{\infty}$
there holds 
\[w_n(x)=u_n(r^{(i)}_nx)+\log(\lambda_n)+c^{(i)}_{\infty}\to w_{\infty}(x)\]
smoothly locally in $\R^2$, where the metric  
$g_{\infty}=e^{2w_{\infty}}g_{\R^2}$ on $\R^2$ has finite volume and finite total curvature
with $K_{g_\infty}(x)=1+(Ax,x)$, where $A=\frac12Hess_{f_0}(p^{(i)}_{\infty})$. 
\end{theorem}

Moreover, we have

\begin{theorem}\label{thm: second main result}
Let $f_0\leq 0$ be a smooth, non-constant function. For $\lambda\in\R$ set
\begin{equation}\label{eqn: l'insieme C lambda}  
\mathcal{C_\lambda} := \left\{ u\in \Hone : \int_M u d\mu_{g_0} = 0 
= \int_M f_\lambda e^{2u} d\mu_{g_0} \right\}.
\end{equation}
Then for any arbitrary sequence $(\lambda_n)_n \subset \Lambda$ such that 
$\lambda_n\uparrow \lambda_{max}$ for $n\to +\infty$, we have that:

i)there exists a sequence of minimizers $w_n\in \mathcal{C}_{\lambda_n}$ of the Dirichlet energy such that:
\[
w_n \to 0 \;\; \mbox{in} \;\; C^{2,\alpha}(M)
\] 
for any $\alpha\in \left[0,1 \right)$.

ii)there exists a sequence of solutions $u_n$ to equation
\[
-\Delta_{g_0}u_n = f_{\lambda_n} e^{2u_n} 
\]
such that $u_n\to -\infty$ uniformly on the whole $M$.
\end{theorem}
\begin{observation}
We remark that in Theorem \ref{thm: second main result} no assumptions have been made on the nature of the points of maximum of the function $f_0$. 
\end{observation}
\begin{observation}
In contrast to \cite{Borer-Galimberti-Struwe}, in the present paper the monotonicity of the energy of the solutions $u_\lambda$ as a function of $\lambda$ is not obvious. The proof of this fact is perhaps the main new technical achievement in the present work.
\end{observation}

\section*{Acknowledgments}
I would like to thank Michael Struwe for the guidance through the project which led to this paper.

\section{Some notation and preliminary results}\label{sec:Some notation and preliminary results}

In the following section we will recall some well-known results about the existence of solutions to equation
(\ref{eqn: the Gauss curvature}) and introduce some notation and concepts used through the rest of the paper.
For further details we refer to \cite{Kazdan-Warner74}.

For $\lambda \in \R$ consider the set $\mathcal{C_\lambda}$ defined by (\ref{eqn: l'insieme C lambda}).
Note that for $\lambda\in (0, -\min_M f_0)$ the function $f_\lambda$ is sign changing and hence 
$\mathcal{C_\lambda} \neq \emptyset$. On the other hand, $\mathcal{C_\lambda} = \emptyset$ for $\lambda \leq 0$ or $\lambda \geq -\min_M f_0$. 

The constraints defining $\mathcal{C_\lambda}$ are natural; the first allows to apply the direct methods, the second one is motivated by the Gauss-Bonnet Theorem. 

\begin{lemma}\label{lemma: C lambda e' una varieta'}
For $\lambda\in (0, -\min_M f_0)$ the set $\mathcal{C_\lambda}$ is a $C^{\infty}$-Banach manifold.
\end{lemma}
\begin{proof}
Define $G^\lambda : \Hone \to \R^2$ by letting
\[
G^\lambda(u) :=\left( \int_M u d\mu_{g_0} , \int_M f_\lambda e^{2u} d\mu_{g_0} \right).
\]
Then $G^\lambda$ is smooth and its first derivative is
\[
DG^\lambda (u)\left[v\right] = \left( \int_M v d\mu_{g_0} , 2\int_M f_\lambda v e^{2u} d\mu_{g_0} \right).
\]
Notice that $(G^\lambda)^{-1}(0)= \mathcal{C_\lambda}$. Pick $u\in \mathcal{C_\lambda}$.
If we compute $DG^\lambda (u)\left[v\right]$ with $v\equiv 1$ and then with $v=f_\lambda$, we get
two vectors of $\R^2$ which are linearly independent; therefore $DG^\lambda(u)$ is surjective.
Since we are in the Hilbert space $\Hone$, we have that it is splitted by the kernel of $DG^\lambda(u)$. It follows that $G^\lambda$ is a submersion at $u\in\mathcal{C_\lambda}$ and then that $\mathcal{C_\lambda}$ is a
smooth manifold. (For further details we refer to \cite{Zeidler}). The lemma is proved.
\end{proof}

In order to find solutions to equation (\ref{eqn: the Gauss curvature}) for $\lambda\in\Lambda$, we minimize the Dirichlet energy $E$ 
\[
\Hone \ni u \stackrel{E}{\longmapsto } \int_M |\nabla u|_{g_0}^2 d\mu_{g_0}
\]
in $\mathcal{C_\lambda}$. The energy $E$ is coercive on
$\mathcal{C_\lambda}$ in view of Poincare's inequality and sequentially weakly lower semicontinuous.
Furthermore, $\mathcal{C_\lambda}$ is weakly sequentially closed as can easily be shown by means of Moser-Trudinger's inequality.
Hence the direct method of the calculus of variation applies and for each $\lambda\in\Lambda$ there
exists a minimizer $w_\lambda\in \mathcal{C_\lambda}$.

But in the course of the proof of Lemma \ref{lemma: C lambda e' una varieta'} we have seen that $G^\lambda$
is a submersion at any point of $\mathcal{C_\lambda}$: therefore we can apply the Lagrange multipliers rule
and obtain
\begin{equation}\label{eqn: moltiplicatori di Lagrange}
2\int_M (\nabla w_\lambda,\nabla v)_{g_0} d\mu_{g_0} = \sigma \int_M v d\mu_{g_0}
+ 2\mu \int_M f_\lambda v e^{2w_\lambda} d\mu_{g_0} 
\end{equation}
for every $v\in\Hone$ with suitable $\sigma,\mu\in\R$. Choosing $v\equiv 1$, we obtain
\[
 0 = \sigma \int_M d\mu_{g_0}
+ 2\mu \int_M f_\lambda  e^{2w_\lambda} d\mu_{g_0} ;
\]
hence $\sigma=0$, because $w_\lambda\in\mathcal{C_\lambda}$. Notice that, by regularity arguments
(see \cite{Kazdan-Warner74} for the details), $w_\lambda\in C^{\infty}(M)$ and hence $v\equiv e^{-2w_\lambda}\in\Hone$. For this choice of testing function (\ref{eqn: moltiplicatori di Lagrange}) gives
\[
0 \geq -2 \int_M |\nabla w_\lambda|_{g_0}^2 e^{-2w_\lambda} d\mu_{g_0} = \mu \int_M f_\lambda d\mu_{g_0}.
\] 
If 
\[
\int_M |\nabla w_\lambda|_{g_0}^2 e^{-2w_\lambda} d\mu_{g_0} =0,
\]
we get $w_\lambda \equiv constant$, which is a contradiction, since in $\mathcal{C_\lambda}$ there are no
constant functions for $\lambda\in\Lambda$. Therefore, since $\int_M f_\lambda d\mu_{g_0}<0$ for
$\lambda\in\Lambda$,  we obtain
\[
\mu=\mu(\lambda) = -2 \frac{\int_M |\nabla w_\lambda|_{g_0}^2 e^{-2w_\lambda} d\mu_{g_0}}
{\int_M f_\lambda d\mu_{g_0}}> 0.
\]
As a consequence,
\begin{equation}\label{eqn: definizione della soluzione u lambda}
 u_\lambda := w_\lambda + 1/2 \log \mu(\lambda) 
\end{equation}
classically solves (\ref{eqn: the Gauss curvature}). 

For the continuation of our analysis and for technical reasons which will become evident later, it is
convenient to introduce for $\lambda\in \R$ the set
\[
\mathcal{E_\lambda} := \left\{ u\in \Hone : 0 = \int_M f_\lambda e^{2u} d\mu_{g_0} \right\}, 
\]
defined by a single constraint only.

As above, it can be seen that $\mathcal{E_\lambda}\neq \emptyset$ if and only if $\lambda\in(0,-\min_M f_0)$ and 
that it is a $C^{\infty}$-Banach manifold.

A priori it is not clear if we may expect that the Dirichlet energy $E$ attains a mimimun in $\mathcal{E_\lambda}$; however an elementary argument shows that for $\lambda\in (0,\lambda_{max})$ we have
\[
E(u_\lambda) = \min_{v\in \mathcal{E_\lambda}} E(v) = \min_{v\in \mathcal{C_\lambda}} E(v) 
\]
where $u_\lambda$ is defined by (\ref{eqn: definizione della soluzione u lambda}). Indeed, for any 
$v\in \mathcal{E_\lambda}$, we have $v-\overline{v}\in \mathcal{C_\lambda}$ and $E(v)=E(v-\overline{v})$,
where $\overline{v}:= \int_M v d\mu_{g_0}$.

Notice finally that for $\lambda= \lambda_{max}$, $u\equiv constant$ belongs to $\mathcal{E_\lambda}$
and it mimimizes the energy (which is zero). Furthermore, for 
$\lambda\in (\lambda_{max}, -\min_M f_0) $, it is always true that the energy $E$, even though it does not admit a mimimum, is non negative. That suggests to define the following function: 
\begin{equation}\label{eqn: definizione di beta lambda}
\beta_\lambda := \left\{
\begin{array}{ll}
\int_M |\nabla u_\lambda|_{g_0}^2 d\mu_{g_0}=E(u_\lambda) & \mbox{if } \lambda \in (0, \lambda_{max}) \\
0 & \mbox{if } \lambda\in \left[\lambda_{max}, -\min_M f_0 \right) .
\end{array}
\right.  
\end{equation}

In the next sections, we study the properties of $\beta_\lambda$ and use this
information to prove, respectively, Theorem \ref{thm: first main result} 
and Theorem \ref{thm: second main result}.

\section{Proof of Theorem \ref{thm: first main result}}

In this section we will analyse the behaviour of the set of solutions to equation 
(\ref{eqn: the Gauss curvature}) when the parameter $\lambda$ approaches zero and we will prove Theorem 
\ref{thm: first main result}. 

The first result is quite elementary but it shows that in an arbitrary neighborhood of zero the function
$\beta_\lambda$ can achieve arbitrarily large values. More precisely, we can state:
\begin{lemma}\label{lemma: the energy blows up}
$\limsup_{\lambda \downarrow 0,\lambda\in\Lambda} \beta_\lambda= +\infty$.
\end{lemma}
\begin{proof}
Assume by contradiction that there exists $\delta\in\Lambda$ such that
$\sup_{\lambda\in(0,\delta)} \beta_\lambda < +\infty$. Choose a sequence $(\lambda_n)_n \subset (0,\delta)$
which converges to zero as $n\to + \infty$. Thus we have $\int_M |\nabla w_{\lambda_n}|_{g_0}^2 d\mu_{g_0}
< +\infty$ uniformly in $n$, where $w_{\lambda_n}\in\mathcal{C}_{\lambda_{n}}$ is a minimizer of the energy
$E$. Therefore, since the average of $w_{\lambda_{n}}$ is zero, we have, up to subsequences, that
$w_{\lambda_{n}}\rightharpoonup w_0$ weakly in $\Hone$ and $e^{2w_{\lambda_{n}}}\rightarrow e^{2w_0}$
strongly in $L^1$. Thus
\[
0= \int_M f_{\lambda_{n}} e^{2w_{\lambda_{n}}} d\mu_{g_0} \rightarrow 
\int_M f_0 e^{2w_0} d\mu_{g_0}
\]
and $w_0\in\mathcal{E}_0=\emptyset$. The contradiction proves the Lemma.
\end{proof}

In the following, we are going to construct a suitable comparison function belonging
to the manifold $\mathcal{E_\lambda}$, which will give a control on the rate of blow-up of the mimimum of the
energy. This is the content of the next proposition, but before we need:
\begin{lemma}\label{lemma: determinazione della costante L}
There exists $L>0$ such that for any $\lambda< -\min_M f_0$ and for any $p\in M$ point of maximum of $f_0$
we have
\begin{enumerate}
  \item $\frac{\sqrt{\lambda}}{L} < 1$
  \item $f_0(x) > -\frac{\lambda}{2}$ on $B_{\frac{\sqrt{\lambda}}{L}}(0) \subset \R^2$, 
\end{enumerate}
where $x$ are suitable local conformal coordinates around $p\simeq 0$. 
\end{lemma}
\begin{proof}
Fix a point of maximum $p_i$ of $f_0$. Then, by choosing local conformal coordinates $x$ around $p_i\simeq 0$, we
have
\[
f_0(x) = \frac{1}{2} D^2f_0(0)[x,x] + O(|x|^3) \;\; \mbox{in}\;\; B_1(0)\subset\R^2
\]
From the beginning we may assume that $\frac{1}{2} D^2f_0(0)[x,x]\geq -c_1 |x|^2$, where $c_1>0$. 
Then, for $x\in B_1(0)$, we have 
\[
f_0(x)\geq -c_1|x|^2 - c_2|x|^3\geq -c(|x|^2 + |x|^3),
\] 
with $c_2>0$ and $c:=\max(c_1,c_2)>0$. 

Pick $\lambda>0$ and $L_i>0$ to be determined later, such that $\sqrt{\lambda}/L_i<1$, namely $\lambda<L_i^2$.
Then, on the ball $B_{\frac{\sqrt{\lambda}}{L_i}}(0)$ we get
\[
f_0(x) > -c\left( \frac{\lambda}{L_i^2} + \frac{\lambda^{3/2}}{L_i^3}\right) \geq -\frac{\lambda}{2}
\]
where the last inequality holds if we choose $L_i^2 \geq 4c$. Choose $L_i>>0$ so that 
$-\min_M f_0 < L_i^2$. Taking $L:= \max_{p_i}L_i$, we obtain the desired result.
\end{proof}

\begin{proposition}\label{prop: construction of the comparison function}
For any $0<\sigma\leq 1$ there exists $\lambda_\sigma <1$, $\lambda_\sigma\in\Lambda$, such that for any 
$0<\lambda\leq\lambda_\sigma$ there holds:
\begin{equation}\label{eqn: stima del blow-up dell'energia}
\beta_\lambda \leq 2\pi M_0 \left(\sigma +2 \right)^2 \log(1/\lambda)  
\end{equation}
where $M_0$ is a constant which depends only on $(M,g_0)$ and the function $f_0$.
\end{proposition}
\begin{proof}
Choose $p_0\in M$ such that $f_0(p_0)=0$ and choose conformal coordinates $x$ as in the previous Lemma so
that
\[
f_0(x) + \lambda \geq \frac{\lambda}{2}, \;\;\; x\in B_{\frac{\sqrt{\lambda}}{L}}(0) 
\]
for any $\lambda < -\min_M f_0$. Locally we can write $g_0= e^{2v_0}g_{\R^2}$ where 
$v_0\in C^{\infty}(\overline{B_1(0)})$ and $v_0(0)=0$. Fix $\lambda\in\Lambda$ with $\lambda<1$.
Define the function $\varphi(\lambda):M\to \R $ as
\begin{equation}\label{eqn: definizione della funzione comparison}
\varphi(\lambda)(x)= 
\left\{
\begin{array}{rl}
\log\left(\frac{\sqrt{\lambda}}{L|x|} \right), & 
\;\; \frac{\lambda^{3/2}}{L} \leq |x| \leq \frac{\sqrt{\lambda}}{L}      \\ 
\log\left(\frac{1}{\lambda} \right), & \;\; |x| \leq \frac{\lambda^{3/2}}{L} \\
0, & \;\; \frac{\sqrt{\lambda}}{L} \leq |x| \leq 1
\end{array}
\right.  
\end{equation}
extended to zero on the rest of $M$. We have $\varphi(\lambda)\in\Hone $ and $f_\lambda$ is positive on the support of $\varphi(\lambda)$. 

Consider the continuous function $z:\R\to\R$ defined by
$z(\alpha)= \int_M f_\lambda e^{2\alpha\varphi(\lambda)} d\mu_{g_0}$; then $z(0)< 0$ and
$\lim_{\alpha\to +\infty}z(\alpha) = +\infty$; thus there exists 
$\alpha=\alpha(\lambda)\in(0,+\infty)$ where 
\[
0 = z(\alpha) =\int_M f_\lambda e^{2\alpha\varphi(\lambda)} d\mu_{g_0},
\] 
that is, $\alpha \varphi(\lambda)\in \mathcal{E_\lambda} $.
 
We can give a more precise estimate of $\alpha$, as follows. Recall that $\vol(M;g_0)=1$, therefore
\begin{eqnarray}
0 = \int_M f_\lambda e^{2\alpha\varphi(\lambda)} d\mu_{g_0} 
  & \geq & \lambda/2\int_{B_{\frac{\sqrt{\lambda}}{L}}(0)} e^{2\alpha\varphi(\lambda) } e^{2v_0}dx
  -||f_0||_{\infty} \nonumber \\
  & > & \lambda/2\int_{B_{\frac{\lambda^{3/2}}{L}}(0)} e^{2\alpha\log(1/\lambda) } e^{2v_0}dx
  -||f_0||_{\infty} . \nonumber
\end{eqnarray}
Let $m_0:= \min_{B_1(0)}e^{2v_0}$ and $M_0:= \max_{B_1(0)}e^{2v_0}$. We obtain:
\[
\frac{m_0\pi}{2} \frac{\lambda^{4-2\alpha}}{L^2} \leq ||f_0||_{\infty} 
\]
or equivalently
\[
0< \alpha \leq \frac{\log\left( \frac{2L^2||f_0||_{\infty}}{m_0\pi}\right)}{2\log(1/\lambda)} + 2.
\]
Given $0<\sigma\leq 1$, there exists $\lambda_\sigma <1$, $\lambda_\sigma\in\Lambda$, such that for any
$0<\lambda\leq\lambda_\sigma$ we have 
$\frac{\log\left( \frac{2L^2||f_0||_{\infty}}{m_0\pi}\right)}{2\log(1/\lambda)} <\sigma$. Hence 
\[
 \alpha^2 \leq \left(\sigma +2 \right)^2.
\]
Next we have:
\[
\int_M |\nabla \varphi (\lambda)|_{g_0}^2 d\mu_{g_0} = 
\int_{B_{\frac{\sqrt{\lambda}}{L}}(0)\setminus B_{\frac{\lambda^{3/2}}{L}}(0)} |x|^{-2} e^{2v_0} dx;
\]
hence
\begin{equation}\label{eqn:estimates for the Dirichlet energy of the comparison function}
m_0 2\pi \log(1/\lambda) \leq \int_M |\nabla \varphi (\lambda)|_{g_0}^2 d\mu_{g_0}
\leq M_0 2\pi \log(1/\lambda).
\end{equation}
We conclude
\[
\beta_\lambda \leq \alpha^2 \int_M |\nabla \varphi (\lambda)|_{g_0}^2 d\mu_{g_0} \leq
2\pi M_0 \left(\sigma +2 \right)^2 \log(1/\lambda),
\]
which proves the Proposition.
\end{proof}
From Proposition \ref{prop: construction of the comparison function}, by means of elliptic estimates we obtain uniform $L^{\infty}$ -bounds for the set of solutions of (\ref{eqn: the Gauss curvature}), away from the boundary of $\Lambda$. More precisely:
\begin{proposition}\label{prop: stime L infinito uniformi}
Fix $0<\sigma\leq 1$ and let $\lambda_\sigma$ be as in Proposition (\ref{prop: construction of the comparison function}). Then for any $\lambda^{\ast}\in(0,\lambda_\sigma)$ we have
\begin{equation}\label{eqn: stime L infinito uniformi}
\sup_{\lambda^{\ast}\leq\lambda\leq\lambda_\sigma} ||u_\lambda||_{\infty} < +\infty .
\end{equation}
\end{proposition}
\begin{observation}
Obviously, the estimate above can be improved by replacing the $L^{\infty}$ norm with "higher" norms (use
a bootstrap argument), but in the rest of the paper the estimate above will turn out to be sufficient for
all our purposes.
\end{observation}
\begin{proof}[Proof of Proposition \ref{prop: stime L infinito uniformi}]
Because of the Sobolev embedding, it is enough to prove that
\[
\sup_{\lambda^{\ast}\leq\lambda\leq\lambda_\sigma} ||u_\lambda||_{H^2} < +\infty
\]
For $\lambda\in \left[\lambda^{\ast},  \lambda_\sigma\right]$ consider the minimizer 
$w_\lambda\in\mathcal{C_\lambda}$, which solves the equation
\[
-\Delta_{g_0}w_\lambda = \mu(\lambda)f_\lambda e^{2w_\lambda}
\]
where $\mu(\lambda)\in(0,\infty)$ is a Lagrange multiplier.

From (\ref{eqn: stima del blow-up dell'energia}), we have
\[
\sup_{{\lambda^{\ast}\leq\lambda\leq\lambda_\sigma}} \beta_\lambda 
\leq C \log(1/\lambda^{\ast}) .
\]
Hence, using Poincar\'e's inequality, we obtain $||w_\lambda||_{\Hone} < C$ uniformly in $\lambda$.
By the Moser-Trudinger's inequality, for every $p\geq 1$ then there holds:
\[
\sup_{{\lambda^{\ast}\leq\lambda\leq\lambda_\sigma}} \int_M e^{pw_\lambda} d\mu_{g_0} < C(p)< \infty.
\]
Our claim thus follows once we can give a lower and an upper bound for $\mu(\lambda)$. Inserting 
$v=f_\lambda$ in (\ref{eqn: moltiplicatori di Lagrange}), we obtain
\begin{equation}\label{eqn:integraz. per parti con test f lambda}
\int_M (\nabla w_\lambda,\nabla f_\lambda)_{g_0} d\mu_{g_0} = \mu(\lambda) \int_M (f_\lambda)^2 e^{2w_\lambda} d\mu_{g_0} .
\end{equation}
Since $\lambda_\sigma < -\overline{f_0}$, we have $0<c \leq \left( \int_M f_\lambda d\mu_{g_0} \right)^2$ uniformly in $\lambda\in \left[\lambda^{\ast}, \lambda_\sigma\right]$. Thus, by H\"older
\[
c< \int_M (f_\lambda)^2 e^{2w_\lambda} d\mu_{g_0} \int_M e^{-2w_\lambda} d\mu_{g_0} .
\]
Applying Moser-Trudinger's inequality, we get
\[
c< C\int_M (f_\lambda)^2 e^{2w_\lambda} d\mu_{g_0} \exp\left(\frac{1}{4\pi}\int_M |\nabla w_\lambda|_{g_0}^2 d\mu_{g_0}\right) .
\]
Thus, we see that $\int_M (f_\lambda)^2 e^{2w_\lambda} d\mu_{g_0} $ for 
$\lambda\in\left[\lambda^{\ast}, \lambda_\sigma\right]$ is uniformly bounded away from zero and, from
(\ref{eqn:integraz. per parti con test f lambda}), we obtain 
\[
\mu(\lambda) \leq C(||w_\lambda||_{\Hone}) < C < \infty
\]
uniformly in $\lambda$.
 
To see that $\mu(\lambda)$ is also away from zero, we argue by contradiction.
Assume that $\inf_{{\lambda^{\ast}\leq\lambda\leq\lambda_\sigma}} \mu (\lambda ) = 0$. Take a sequence $\lambda_n\in \left[\lambda^{\ast},\lambda_\sigma   \right] $ such that:
\[
\mu(\lambda_n ) \to 0
\]
and $\lambda_n \to \lambda\in \left[\lambda^{\ast},\lambda_\sigma   \right]$. From the estimates above,
we can assume, up to subsequences, that $w_{\lambda_n}\rightharpoonup w$ weakly in $\Hone$ and
$e^{2w_{\lambda_n}}\rightarrow e^{2w}$ strongly in $L^1$ as $n\to \infty$. Recall that we have
\[
\int_M (\nabla w_{\lambda_n},\nabla v)_{g_0} d\mu_{g_0} = 
 \mu(\lambda_n) \int_M f_{\lambda_n} v e^{2w_{\lambda_n}} d\mu_{g_0} 
\]
for any $v\in\Hone$. Passing to the limit $n\to\infty$ in this equation, we obtain 
\[
\int_M (\nabla w,\nabla v)_{g_0} d\mu_{g_0} = 0
\]
for each $v\in\Hone$, that is $w$ is harmonic. But then $w\equiv 0$ which is clearly impossible. Therefore, we have shown that for $\lambda\in \left[\lambda^{\ast},\lambda_\sigma \right]$, $\mu(\lambda)$ 
is uniformly away from 0 and infinity.

In conclusion, we get a uniform bound in $\lambda$ for 
\[
||\Delta_{g_0}w_{\lambda}||_{L^2} = ||\mu(\lambda)f_\lambda e^{2w_\lambda}||_{L^2}.
\]
Hence, by $L^p$-elliptic estimates (see for instance \cite{Kazdan-Warner74}, p. 24), we have
\[
||w_\lambda||_{H^2} \leq C \left\{ ||w_\lambda||_{H^1} +  ||\Delta_{g_0}w_{\lambda}||_{L^2}\right\} < C,
\]
uniformly for $\lambda\in \left[\lambda^{\ast},\lambda_\sigma \right]$. Recalling equation
(\ref{eqn: definizione della soluzione u lambda}) and the bounds on $\mu(\lambda)$, the bound
(\ref{eqn: stime L infinito uniformi}) follows.
\end{proof}

\begin{remark}\label{rmk: i massimi delle soluzioni tendono ad esplodere}
The Proposition above is false when $\lambda$ approaches zero. Indeed, an estimate like
$\sup_{0<\lambda\leq\delta} \max_M u_{\lambda} < \infty$ for some $\delta$ would lead,
in view of Schauder's estimates, to a uniform $C^{2,\alpha}$ bound for $u_\lambda$, which clearly contradicts Lemma \ref{lemma: the energy blows up}.
\end{remark}

In the following we show that the function $\beta_\lambda$ is monotone decreasing in a suitable right 
neighborhood of zero, which is crucial for our argument. As a consequence, $\beta_\lambda$ will be differentiable almost everywhere. 
\begin{proposition}\label{prop: teorema strumentale per il teorema di monotonicita'}
There exists $\lambda_0 \leq \min\left\{1/2, -\overline{f_0}/2 \right\}$ such that for any $\lambda^{\ast}\in (0,\lambda_0)$ there exists $\ell(\lambda^\ast)\in (\lambda^{\ast}, -\min_M f_0)$ such that for any 
$\lambda\in (\lambda^\ast, \ell(\lambda^\ast)) $ we have
\[
\beta_\lambda < \beta_{\lambda^{\ast}}
\]
Furthermore, choosing $\lambda\in(0,\lambda_0), \lambda > \lambda^{\ast} $, 
and defining $\ell(\lambda)$ as above, we have 
$\ell(\lambda) - \lambda\geq \tau=\tau(\lambda^{\ast})>0 $ where $\tau$ is a constant not depending on $\lambda$.
\end{proposition}
\begin{corollary}\label{cor: monotonicita' di beta lambda}
There exists $\lambda_0 \leq \min\left\{1/2, -\overline{f_0}/2 \right\}$ such that $\beta_\lambda$ is strictly monotone decreasing on the interval $(0,\lambda_0)$.  
\end{corollary}

In order to prepare for the proof of Proposition \ref{prop: teorema strumentale per il teorema di monotonicita'}, define the map $I:\Hone \to \R$ by letting
\begin{equation}\label{eqn: definizione della mappa I}
I(u) := -\frac{\int_M f_0 e^{2u} d\mu_{g_0}}{\int_M e^{2u} d\mu_{g_0}}.  
\end{equation}
Note that for any $u\in\Hone$ there holds
\begin{equation}\label{eqn:proprieta' cardine della mappa I}
u\in\mathcal{E}_{I(u)}  .
\end{equation}
Moreover, we have $I(u)\in(0,-min_M f_0) $ and $I$ is smooth with first derivative given by the following expression:
\begin{equation}\label{eqn: differenziale di I}
DI(u)[v] = -2 \,\frac{\int_M f_{I(u)} v\,e^{2u} d\mu_{g_0}}{\int_M e^{2u} d\mu_{g_0}}\;\;\;\; u,v\in\Hone . 
\end{equation}
Fix $0 < \lambda_0 \leq \min\left\{1/2, -\overline{f_0}/2 \right\}$ and for $\lambda^{\ast}\in(0,\lambda_0)$
let
\begin{equation}\label{eqn: prima funzione ausiliaria}
A(\lambda^{\ast}) := \sup_{\lambda^{\ast}\leq\lambda <\lambda_0} \sup_M e^{2u_\lambda}  
\end{equation}
and
\begin{equation}\label{eqn: seconda funzione ausiliaria}
a(\lambda^{\ast}) := \inf_{\lambda^{\ast}\leq\lambda <\lambda_0} \inf_M e^{2u_\lambda}.  
\end{equation}
Observe that in view of Proposition \ref{prop: stime L infinito uniformi}, the above functions are
well defined if $\lambda_0$ is taken small enough, and that $0< a(\lambda^\ast)\leq A(\lambda^{\ast})< +\infty$.
Finally, $A$ is monotone decreasing and $a$ is monotone increasing in $\lambda^{\ast}$.

We are ready to prove our Proposition.
\subsection{Proof of Proposition \ref{prop: teorema strumentale per il teorema di monotonicita'}}
\begin{proof}
Fix for convenience $\sigma=1$ and let $\lambda_\sigma$ as given in Proposition \ref{prop: construction of the comparison function}. Consider $\min\left\{\lambda_0, \lambda_\sigma \right\}$, which with a little
abuse of notation we will still call $\lambda_0$.

We consider $\lambda^{\ast}\in(0,\lambda_0)$ and 
$\beta_{\lambda^{\ast}}= \int_M |\nabla u^{\ast}|_{g_0}^2 d\mu_{g_0} $, where we have used the abbreviation
$u^{\ast}\equiv u_{\lambda^{\ast}}$. We also set $\varphi^{\ast}\equiv\varphi(\lambda^{\ast})$, where
$\varphi(\lambda^{\ast})$ is the comparison function defined by the equation 
(\ref{eqn: definizione della funzione comparison}). (We recall that $\lambda^{\ast} <\lambda_0\leq 1/2<1$, therefore
$\varphi^{\ast}$ is well defined.) Thus, we have inequality 
(\ref{eqn:estimates for the Dirichlet energy of the comparison function}) and
\begin{eqnarray*}
\int_M (\nabla u^{\ast},\nabla \varphi^{\ast})_{g_0} d\mu_{g_0} & = &
\int_M f_{\lambda^{\ast}} \varphi^{\ast}e^{2u^{\ast}} d\mu_{g_0}  \\
& > & \frac{\lambda^{\ast}}{2} \log(1/\lambda^{\ast}) \int_{B_{\frac{(\lambda^\ast)^{3/2}}{L}}(0)} e^{2u^{\ast} }e^{2v_0}dx ,\nonumber
\end{eqnarray*}
since $f_{\lambda^{\ast}} \varphi^{\ast} \geq 0$ and since $f_{\lambda^{\ast}} \geq \lambda^{\ast}/2$ in the ball $B_{\frac{\sqrt{\lambda^{\ast}}}{L}}(0)$.
Observing that
\begin{equation}\label{eqn: stima dell'integrale di lambda star alla 3/2}
\int_{B_{\frac{(\lambda^\ast)^{3/2}}{L}}(0)}e^{2u^{\ast} }e^{2v_0} dx
\geq \frac{m_0\pi}{L^2} (\lambda^{\ast})^3 a(\lambda^{\ast})  
\end{equation}
where $a(\lambda^\ast)$ is defined by (\ref{eqn: seconda funzione ausiliaria}) and $m_0=\min_{B_1(0)}e^{2v_0}$ as above, we obtain
\begin{eqnarray}\label{eqn: l'integrale che coinvolge phi star e' positivo}
\int_M (\nabla u^{\ast},\nabla \varphi^{\ast})_{g_0} d\mu_{g_0} & = &
\int_M f_{\lambda^{\ast}} \varphi^{\ast}e^{2u^{\ast}} d\mu_{g_0}  \\
& > & \frac{m_0\pi}{2L^2}\log(1/\lambda^{\ast})(\lambda^{\ast})^4  a(\lambda^\ast) \; >0 . \nonumber
\end{eqnarray}
Moreover, using equations (\ref{eqn: stima del blow-up dell'energia}) and 
(\ref{eqn:estimates for the Dirichlet energy of the comparison function}), from H\"older's inequality we deduce
\[
\int_M (\nabla u^{\ast},\nabla \varphi^{\ast})_{g_0} d\mu_{g_0}\leq 6\pi M_0 \log(1/\lambda^{\ast}).
\]   
Hence, defining 
\begin{equation}\label{eqn: definzione di epsilon star}
 \varepsilon^{\ast} := 2 \frac{\int_M (\nabla u^{\ast},\nabla \varphi^{\ast})_{g_0} d\mu_{g_0}}
 {\int_M |\nabla \varphi^{\ast}|_{g_0}^2 d\mu_{g_0}} 
\end{equation}
and using inequality (\ref{eqn: l'integrale che coinvolge phi star e' positivo}) and once more
(\ref{eqn:estimates for the Dirichlet energy of the comparison function}), we eventually get 
\begin{equation}\label{eqn: bounds for epsilon star}
\frac{m_0}{2 M_0 L^2}(\lambda^{\ast})^4 a(\lambda^{\ast}) < \varepsilon^{\ast} < \frac{6M_0}{m_0} .
\end{equation}
In particular, $\varepsilon^{\ast}$ is positive. (Recall that $M_0:= \max_{B_1(0)}e^{2v_0}$).

For $\varepsilon\in [-\varepsilon^{\ast}, \varepsilon^{\ast} ]$ consider the function
$u^{\ast}-\varepsilon \varphi^{\ast}\in\Hone$. Recall that by (\ref{eqn:proprieta' cardine della mappa I}), 
we trivially have
\[
u^{\ast}-\varepsilon \varphi^{\ast}\in\mathcal{E}_{I(u^{\ast}-\varepsilon \varphi^{\ast})} .
\]
\begin{lemma}\label{lemma: I risultato di monotonicita' di beta lambda}
For $\varepsilon\in (0, \varepsilon^{\ast} )$ we have
\begin{equation}\label{eqn: prima stima di monotonicita' sui beta}
\beta_{I(u^{\ast}-\varepsilon \varphi^{\ast})} < \beta_{\lambda^{\ast}} .
\end{equation}
\end{lemma}
\begin{proof}
By expanding the Dirichlet energy, for $\varepsilon\in (0, \varepsilon^{\ast} )$ we obtain
\begin{eqnarray*}
\beta_{I(u^{\ast}-\varepsilon \varphi^{\ast})} \leq E(u^{\ast}-\varepsilon \varphi^{\ast}) & = &
E(u^{\ast}) -2\varepsilon \int_M (\nabla u^{\ast},\nabla \varphi^{\ast})_{g_0} d\mu_{g_0} 
+\varepsilon^2 \int_M |\nabla \varphi^{\ast}|_{g_0}^2 d\mu_{g_0} \\
& = &
E(u^{\ast}) -\varepsilon(\varepsilon^{\ast} - \varepsilon) \int_M |\nabla \varphi^{\ast}|_{g_0}^2 d\mu_{g_0} \\
& < & E(u^{\ast}) = \beta_{\lambda^{\ast}}, 
\end{eqnarray*}
as claimed.
\end{proof}
The next step is to understand whether the value $I(u^{\ast}-\varepsilon \varphi^{\ast})$ is greater or
smaller than $\lambda^{\ast}=I(u^{\ast})$. In order to do that, we introduce the function
$h:[-\varepsilon^{\ast}, \varepsilon^{\ast} ]\to (0,-min_M f_0)$
given by:
\begin{equation}\label{eqn: funzione ausiliaria h}
  h(\varepsilon ) := I(u^{\ast}-\varepsilon \varphi^{\ast}) .
\end{equation}
By definition of $I$, we have $h\in C^1(\left[-\varepsilon^{\ast},\varepsilon^{\ast} \right])$; moreover,
there holds:
\begin{lemma}\label{lemma: proprieta' di h}
We have that
\[
h'>0 \;\; \mbox{on} \;\; [0,\varepsilon^{\ast}] .
\]
As a consequence, $h$ is smoothly invertible on $[0, \varepsilon^{\ast}]$. 
\end{lemma}
Postponing the proof of the lemma, we continue with the proof of Proposition
\ref{prop: teorema strumentale per il teorema di monotonicita'}.

In view of Lemma \ref{lemma: proprieta' di h}, we have $h(\varepsilon^{\ast}) > \lambda^{\ast}$. 
Furthermore, for any $\lambda\in(\lambda^{\ast}, h(\varepsilon^{\ast} ))$ there exists a unique
$\varepsilon \in (0, \varepsilon^{\ast} )$ such that 
$h(\varepsilon )= I(u^{\ast}-\varepsilon \varphi^{\ast})=\lambda$. From Lemma
\ref{lemma: I risultato di monotonicita' di beta lambda}, then we get
$\beta_\lambda < \beta_{\lambda^{\ast}}$.

Therefore, setting $\ell(\lambda^\ast) := h(\varepsilon^{\ast})$, we obtain the first part of Proposition
\ref{prop: teorema strumentale per il teorema di monotonicita'}.

It remains to show the estimate on the length of this interval $(\lambda^{\ast}, \ell(\lambda^\ast))$
and the relations between it and $(\lambda,\ell(\lambda))$, for $\lambda>\lambda^{\ast}$.
This will be done in Lemma \ref{lemma: lunghezza dell'intervallo U lambda star e varie}.

\begin{proof}[Proof of Lemma \ref{lemma: proprieta' di h}]
Recall that $h(0)=\lambda^{\ast}$. Compute the first derivative of $h$, using 
(\ref{eqn: differenziale di I}):
\[
h'(\varepsilon) = DI(u^{\ast}-\varepsilon \varphi^{\ast})[- \varphi^{\ast}]=
2\,\frac{\int_M f_{I(u^{\ast}-\varepsilon \varphi^{\ast}   )}
 \varphi^{\ast} \,e^{2u^{\ast}-2\varepsilon \varphi^{\ast}} d\mu_{g_0}}
 {\int_M e^{2u^{\ast}-2\varepsilon \varphi^{\ast}} d\mu_{g_0}} .
\]
Thus
\[
h'(0) = 2\,\frac{\int_M f_{\lambda^{\ast}}
 \varphi^{\ast} \,e^{2u^{\ast}} d\mu_{g_0}}
 {\int_M e^{2u^{\ast}} d\mu_{g_0}} >0
\]
in view of (\ref{eqn: l'integrale che coinvolge phi star e' positivo}).
By continuity of $h'$, there exists $\varepsilon\in (0,\varepsilon^{\ast} ]$ such that 
$h'>0$ on $[0,\varepsilon )$ and such that $\varepsilon$ is maximal with this property. We claim that
$\varepsilon=\varepsilon^{\ast}$. Suppose by contradiction that $\varepsilon<\varepsilon^{\ast}$. Note
that $h(\varepsilon ) = I(u^{\ast}-\varepsilon \varphi^{\ast})>\lambda^{\ast}=h(0)$, since $h'>0$ on
$[0,\varepsilon )$. Moreover,
\[
\begin{split}
\int_M f_{h(\varepsilon )}\,
 \varphi^{\ast} \,e^{2u^{\ast}-2\varepsilon \varphi^{\ast}} d\mu_{g_0} & = 
 \int_{B_{\frac{\sqrt{\lambda^\ast}}{L}}(0) } f_{h(\varepsilon)}\,
 \varphi^{\ast} \,e^{2u^{\ast}-2\varepsilon \varphi^{\ast}} e^{2v_0}dx\\
 &  \geq \frac{h(\varepsilon)}{2} \,
 \int_{B_{\frac{\sqrt{\lambda^\ast}}{L}}(0)} 
 \varphi^{\ast}  \,e^{(2u^{\ast}-2\varepsilon \varphi^{\ast}) } 
 e^{2v_0}dx
\end{split}
\]
where in the last inequality we used the fact that 
\[
f_{h(\varepsilon )}\geq \frac{h(\varepsilon)}{2} \;\;
\mbox{on} \;\; B_{\frac{\sqrt{h(\varepsilon)}}{L}}(0) \supset  B_{\frac{\sqrt{\lambda^{\ast}}}{L}}(0) 
\]
(recall Lemma \ref{lemma: determinazione della costante L}). Therefore, we obtain
\[
\begin{split}
\int_M f_{h(\varepsilon )}\,
 \varphi^{\ast} \,e^{2u^{\ast}-2\varepsilon \varphi^{\ast}} d\mu_{g_0} & \geq 
 \frac{h(\varepsilon )}{2} \log(1/\lambda^{\ast}) \,
 \int_{B_{\frac{(\lambda^\ast)^{3/2}}{L}}(0)} e^{(2u^{\ast}-2\varepsilon \varphi^{\ast}) } 
 e^{2v_0}dx \\
 & > 
 \frac{\lambda^{\ast}}{2} \log(1/\lambda^{\ast}) \,
 \int_{B_{\frac{(\lambda^\ast)^{3/2}}{L}}(0)} e^{(2u^{\ast}-2\varepsilon^{\ast} \varphi^{\ast}) } 
 e^{2v_0}dx \\
 & =
 \frac{(\lambda^{\ast})^{1+2\epsilon^{\ast} }}{2} \log(1/\lambda^{\ast}) \,
 \int_{B_{\frac{(\lambda^\ast)^{3/2}}{L}}(0)} e^{2u^{\ast} } 
 e^{2v_0}dx \\
 & \geq
 \frac{m_0\pi}{2L^2} \log(1/\lambda^{\ast}) (\lambda^\ast)^{4 +2\varepsilon^\ast} a(\lambda^\ast) \;>0
\end{split}
\]  
where in the last line we used (\ref{eqn: stima dell'integrale di lambda star alla 3/2}).
Thus, we have, since $\varepsilon>0$ and $\varphi^{\ast}\geq 0 $, 
\begin{eqnarray}\label{eqn: prima stima della derivata di h}
h'(\varepsilon ) & > & \frac{m_0\pi}{L^2} \log(1/\lambda^{\ast})(\lambda^{\ast})^{4+2\epsilon^{\ast} } 
\frac{a(\lambda^\ast)}{\int_M e^{2u^{\ast}-2\varepsilon \varphi^{\ast}} d\mu_{g_0}} \nonumber \\
& > & 
\frac{m_0\pi}{L^2}\log(1/\lambda^{\ast}) (\lambda^{\ast})^{4+2\epsilon^{\ast} } 
\frac{a(\lambda^\ast)}{\int_M e^{2u^{\ast}} d\mu_{g_0}} > 0,
\end{eqnarray}
contradicting the maximality of $\varepsilon$. 
Furthermore, reasoning as we have just done, we see that the bound (\ref{eqn: prima stima della derivata di h})
holds uniformly on $(0,\varepsilon^{\ast})$. We deduce $h'(\varepsilon^{\ast}) >0$ and the Lemma is proved.
\end{proof}
\begin{lemma}\label{lemma: lunghezza dell'intervallo U lambda star e varie}
Let $\lambda_0$ be defined as in the proof of Proposition 
\ref{prop: teorema strumentale per il teorema di monotonicita'}.
Fix $0<\lambda^\ast < \lambda < \lambda_0$ and consider $\ell(\lambda)$ given by the first part of Proposition \ref{prop: teorema strumentale per il teorema di monotonicita'}. Then
$$
\ell(\lambda) - \lambda \geq \tau = \tau(\lambda^\ast) >0 
$$
where $\tau$ is a constant not depending on $\lambda$.
\end{lemma}
\begin{proof}[Proof of Lemma \ref{lemma: lunghezza dell'intervallo U lambda star e varie}]
Let's begin with estimating $\ell(\lambda^\ast) - \lambda^{\ast}$. We restart from 
(\ref{eqn: prima stima della derivata di h}), which holds for $\varepsilon\in(0,\varepsilon^{\ast})$.
By equation (\ref{eqn: bounds for epsilon star}) and by the fact that $\lambda^{\ast} < 1$, we get
$(\lambda^{\ast})^{4+2\epsilon^{\ast} } > (\lambda^{\ast})^{4+ \frac{12M_0}{m_0} }$ and
$\log(1/\lambda^{\ast})> \log(1/\lambda_0) $. Recalling the definition of the auxiliary function
$A$ (equation (\ref{eqn: prima funzione ausiliaria})), we can bound
\[
A(\lambda^{\ast}) \geq \int_M e^{2u^{\ast}} d\mu_{g_0} 
\]
and obtain
\[
h'(\varepsilon) >\frac{ m_0\pi}{L^2} \log(1/\lambda_0) (\lambda^{\ast})^{4+ \frac{12M_0}{m_0}} 
\frac{a(\lambda^{\ast})}{A(\lambda^{\ast})} .
\]
Recalling once more (\ref{eqn: bounds for epsilon star}), with the constant
$k_0:=\frac{m_0^2\pi}{2M_0L^4}\log(1/\lambda_0)>0 $ we may finally estimate 
\begin{eqnarray*}
\ell(\lambda^\ast) -\lambda^{\ast} = h(\varepsilon^{\ast})-\lambda^{\ast} & = & 
\int_0^{\varepsilon^{\ast}} h'(\varepsilon)d\varepsilon \\
&> & \varepsilon^\ast \frac{ m_0\pi}{L^2} \log(1/\lambda_0) (\lambda^{\ast})^{4+ \frac{12M_0}{m_0}} 
\frac{a(\lambda^{\ast})}{A(\lambda^{\ast})} \\
&>&
k_0 (\lambda^{\ast})^{8+ \frac{12M_0}{m_0}}
\frac{(a(\lambda^{\ast}))^2}{A(\lambda^{\ast})} ,
\end{eqnarray*}
and the function $(\lambda^{\ast})^{8+ \frac{12M_0}{m_0}}\frac{(a(\lambda^{\ast}))^2}{A(\lambda^{\ast})}$
is not decreasing in $\lambda^{\ast}$. 

Hence, taking $\lambda\in (\lambda^{\ast}, \lambda_0)$, we deduce
\[
\ell(\lambda) - \lambda
> k_0 (\lambda)^{8+ \frac{12M_0}{m_0}}\frac{(a(\lambda))^2}{A(\lambda)}
\geq k_0 (\lambda^{\ast})^{8+ \frac{12M_0}{m_0}}\frac{(a(\lambda^{\ast}))^2}{A(\lambda^{\ast})}
:= \tau>0 .
\]
The Lemma is proved.
\end{proof}
This concludes the proof of Proposition \ref{prop: teorema strumentale per il teorema di monotonicita'}.
\end{proof}
\subsection{A bound for the total curvature}\mbox{}

With the help of Corollary \ref{cor: monotonicita' di beta lambda},
Proposition \ref{prop: construction of the comparison function} and following \cite{Struwe93}, 
it is now quite straightforward to show the following estimate for the derivative of $\beta_\lambda$:
\begin{lemma}\label{lemma: stima per la derivata di beta lambda}
There exists a sequence $(\lambda_n)_n\subset (0,\lambda_0)$ of points of differentiability for $\beta_\lambda$, such
that $\lambda_n \downarrow 0$ as $n\rightarrow \infty$ and 
\[
|\beta'_{\lambda_n} | \leq C_0/\lambda_n
\]
where $C_0$ is a positive constant.
\end{lemma}
\begin{proof}
By Proposition \ref{prop: construction of the comparison function}, we have
$\beta_\lambda \leq C \log(1/\lambda)$, for any $\lambda<\lambda_0$. Set $C_0:=C+1$ and assume that
exists $\tilde{\lambda}< \lambda_0$ such that for any $\lambda<\tilde{\lambda}$, $\lambda$ point
of differentiability of $\beta_\lambda$, there holds:
\[
|\beta'_{\lambda}| > C_0/\lambda \, .
\]
Then we obtain, by Lebesgue's Theorem, that
\[
\beta_\lambda - \beta_{\tilde{\lambda}} \geq \int^{\tilde{\lambda}}_\lambda |\beta'_{s}|ds
\]
and hence 
\[
C\log(1/\lambda)\geq \beta_\lambda > \beta_{\tilde{\lambda}} + C_0 \log(\tilde{\lambda}/\lambda) .
\]
Thus, we get
\[
\beta_{\tilde{\lambda}} + C_0 \log(\tilde{\lambda}) -\log(\lambda) \leq 0 ,
\]
which, for $\lambda$ small enough, is clearly impossible. The Lemma is proved.
\end{proof}

We can now prove the analogue of equation (5.1) in \cite{Borer-Galimberti-Struwe}: 
\begin{proposition}\label{prop: bound for lambda x volume}
Let $(\lambda_n)_n$ be a sequence like the one given by Lemma \ref{lemma: stima per la derivata di beta lambda}.
and set $u_n:= u_{\lambda_n}$. Then
\begin{equation}\label{eqn: bound for lambda x volume}
  \limsup_n \left( \lambda_n \int_M e^{2u_n} d\mu_{g_0} \right) < \infty
\end{equation}
\end{proposition}
\begin{proof}
Fix $n\in\N$ and set for convenience $\lambda^{\ast}:=\lambda_n
\in(0,\lambda_0)$ and $u^{\ast}:=u_n$.

Consider the function $h$ defined by equation (\ref{eqn: funzione ausiliaria h}), where
$\varepsilon^{\ast}$ and $\varphi^{\ast}$ are defined as in the proof of Proposition
\ref{prop: teorema strumentale per il teorema di monotonicita'}.
For $\lambda_k \downarrow \lambda^{\ast}$, $\lambda_k<h(\varepsilon^{\ast})$, set
$\varepsilon_k := h^{-1}(\lambda_k)$. By Lemma \ref{lemma: proprieta' di h}, 
$\varepsilon_k \to 0$ as $k\to \infty$. Finally, by Lemma \ref{lemma: stima per la derivata di beta lambda},
we may assume that for all $k$
\[
-\frac{\beta_{\lambda^{\ast}}-\beta_{\lambda_k} }{\lambda^{\ast}-\lambda_k} \leq 2C_0/\lambda^{\ast} :=
C/\lambda^{\ast}
\]
where $C_0$ is the constant of Lemma \ref{lemma: stima per la derivata di beta lambda}.

Observe that $\beta_{\lambda^{\ast}}-\beta_{\lambda_k} \geq
E(u^{\ast})- E(u^{\ast}-\varepsilon_k\varphi^{\ast})$, since 
$u^{\ast}-\varepsilon_k\varphi^{\ast}\in\mathcal{E}_{I(u^{\ast}-\varepsilon_k\varphi^{\ast})}=
\mathcal{E}_{\lambda_k}$. Now:
\[
E(u^{\ast})- E(u^{\ast}-\varepsilon_k\varphi^{\ast})=
\int_M \left(-\varepsilon_k^2 |\nabla \varphi^{\ast}|_{g_0}^2 + 2\varepsilon_k(\nabla u^{\ast},\nabla 
\varphi^{\ast} )_{g_0}\right) d\mu_{g_0} .
\]
Hence,
\[
\frac{C}{\lambda^{\ast}} \geq \frac{1}{\lambda_k -\lambda^{\ast}} 
\int_M \left(-\varepsilon_k^2 |\nabla \varphi^{\ast}|_{g_0}^2 + 2\varepsilon_k(\nabla u^{\ast},\nabla 
\varphi^{\ast} )_{g_0} \right)d\mu_{g_0} .
\]
Recalling that $\varepsilon_k = h^{-1}(\lambda_k)$, $h^{-1}(\lambda^{\ast})=0$ and using
(\ref{eqn:estimates for the Dirichlet energy of the comparison function}), we have
\begin{eqnarray*}
\frac{1}{\lambda_k -\lambda^{\ast}}\int_M \varepsilon_k^2 |\nabla \varphi^{\ast}|_{g_0}^2
&\leq & 2\pi M_0\log(1/\lambda^{\ast}) \frac{h^{-1}(\lambda_k)- h^{-1}(\lambda^{\ast})}{\lambda_k 
-\lambda^{\ast}} \varepsilon_k \\
& \to & 0
\end{eqnarray*}
as $k\rightarrow \infty$, since $h^{-1}$ is differentiable at $\lambda^{\ast}$ and $\varepsilon_k$ goes to
zero. Therefore, we may write, with an error term $o(1)$ as $k\rightarrow \infty$, that
\[
\frac{C}{\lambda^{\ast}} \geq 
\frac{2\varepsilon_k}{\lambda_k -\lambda^{\ast}} \int_M (\nabla u^{\ast},\nabla 
\varphi^{\ast} )_{g_0} d\mu_{g_0} + o(1) .
\] 
Thus, when $k\to\infty$, we obtain
\begin{eqnarray*}
\frac{C}{\lambda^{\ast}} &\geq & 
2(h^{-1})'(\lambda^{\ast}) \int_M (\nabla u^{\ast},\nabla 
\varphi^{\ast} )_{g_0} d\mu_{g_0}\\
&=& \frac{2}{h'(0)} \int_M f_{\lambda^{\ast}}
 \varphi^{\ast} \,e^{2u^{\ast}} d\mu_{g_0} \\
&=& \int_M e^{2u^\ast} d\mu_0 
\end{eqnarray*}
where in the last line we have used the explicit expression of $h'(0)$. Going back to the original
notation, we have for any $n\in\N$
\[
\int_M e^{2u_n} d\mu_{g_0} \leq C/\lambda_n ,
\]
which is nothing but equation (\ref{eqn: bound for lambda x volume}). The Proposition is proved. 
\end{proof}

As a consequence of Proposition \ref{prop: bound for lambda x volume} and the Gauss-Bonnet identity
$0= \int_M f_{\lambda_n} e^{2u_n} d\mu_{g_0}$, we deduce the uniform bound
\[
\sup_{n\in\N} \int_M (|f_0|+\lambda_n)e^{2u_n} d\mu_{g_0} < \infty
\] 
for the total curvature of $g_n=e^{2u_n}g_0$.

\subsection{Blow-up analysis}\mbox{}

In this subsection we complete the Proof of Theorem \ref{thm: first main result}. For the rest of this part,
let $(\lambda_n)_n$ be a sequence like the one given by Lemma \ref{lemma: stima per la derivata di beta lambda} and set $u_n:=u_{\lambda_n}$. We follow closely Section 5 of \cite{Borer-Galimberti-Struwe}.

As shown by Ding-Liu \cite{Ding-Liu95}, we obtain for any open domain 
$\Omega\subset\subset M^-:=\left\{p\in M: f_0(p)< 0 \right\}$,
$\int_\Omega (|\nabla u_n^+|^2_{g_0} + |u_n^+|^2) d\mu_{g_0} \leq C(\Omega)$, where $t^+ =\max
\left\{t,0 \right\}$, $t\in\R$,
and hence, as proved in \cite{Borer-Galimberti-Struwe}, that
\begin{equation}\label{eqn:local pointwise upper bound for the solutions}
u_n \leq C'(\Omega) .
\end{equation}
Thus, if a sequence $(u_n)_n$ blows up near a point $p_0\in M$ in the sense that for every
$r>0$ there holds $\sup_{B_r(p_0)} u_n\to +\infty$ (and we know that it is always the case
in view of Remark \ref{rmk: i massimi delle soluzioni tendono ad esplodere}),
necessarily $f_0(p_0)=0$. Moreover, there exists a sequence of points $p_n\to p_0$ such that
for some $r>0$, $u_n(p_n)=\sup_{B_r(p_0)}u_n$. 

Let $p_0$ be such a blow-up point for a sequence of solutions $u_n$. We introduce local isothermal coordinates $x$ on $B_r(p_0)$ around $p_0=0$. We can write $g_0 = e^{2v_0}g_{\R^2}$ for some smooth function $v_0$. Setting
$v_n:= u_n + v_0$, we get
\[
-\Delta v_n = (f_0(x) + \lambda_n)e^{2v_n} \;\;\; \mbox{on} \;\;\; B_R(0)
\]
for some $R>0$ and there is a sequence $x_n\to 0$ so that
\[
v_n(x_n)=\sup_{|x|\leq R}v_n(x)\to +\infty
\]
as $n\to +\infty$. Moreover, $\Delta v_n(x_n)\leq 0$ and thus $f_0(x_n)+\lambda_n\geq 0$, which leads to
\[
|x_n|^2 \leq C \lambda_n
\]
for some constant $C>0$.

We observe that in the present case we do not have available a uniform global lower bound for the
sequence of solutions $u_n$ (and hence for $v_n$) of the kind present in \cite{Borer-Galimberti-Struwe}. 
But we can still show that the analogue of Lemma 5.2 \cite{Borer-Galimberti-Struwe} holds
true. Indeed, a careful inspection shows that a uniform lower bound is not needed in the proof of
Lemma 5.2 \cite{Borer-Galimberti-Struwe}.
\begin{lemma}
For every $r>0$, that holds 
\[
\limsup_n \int_{B_r(0)} (f_0+\lambda_n)^+ e^{2v_n} dx \geq 2\pi \, .
\]
\end{lemma} 

In order to prove Theorem \ref{thm: first main result}, as regards part ii),
we would like to imitate the proof of Theorem 1.4 \cite{Borer-Galimberti-Struwe}. To do that and to show the
 convergence results therein, the last ingredient we need is at least a local lower bound for our sequence 
 of solutions $u_n$. 
 
The next Lemma shows that either the sequence degenerates or that we have a local lower bound. 
After this Lemma, we will obtain part i) of Theorem \ref{thm: first main result}. To prove part ii), it will 
be sufficient to repeat the same reasoning as after Lemma 5.2. in \cite{Borer-Galimberti-Struwe}. 

\begin{lemma}\label{lemma: local lower bound for the sequence of solutions or convergence to minus inf}
Let $(\lambda_n)_n$ and $(u_n)_n$ be defined as above and set
\[
M_\infty:= M \setminus  \left\{p_{\infty}^{(1)}, \cdots ,p_{\infty}^{(I)} \right\} .
\] 
where $p_{\infty}^{(1)}, \cdots ,p_{\infty}^{(I)}$ are blow-up points. Then, up to subsequences, either\\
i) $u_n \to -\infty$ locally uniformly on compact domains of $M_\infty$, or\\
ii) for any compact domain $\Omega\subset\subset M_\infty$, there exists a constant $C=C(\Omega)\in\R$ such that
\[
u_n\big |_{\Omega} > C(\Omega)
\]
uniformly in $n$.
\end{lemma}
\begin{proof}
We fix two open domains $\Omega\subset\subset\tilde{\Omega}\subset\subset M_\infty$. From 
(\ref{eqn:local pointwise upper bound for the solutions}),
for any $n$ we get that $u_n\big |_{\tilde{\Omega}} \leq C(\tilde{\Omega})$. 
We pick an arbitrary point $p\in \overline{\Omega}$ and $r_p>0$ so that 
$B_{r_p}(p)\subset\tilde{\Omega}$. If needed, we choose a smaller radius and we consider a conformal chart
$\Psi: B_{r_p}(p) \to B_1(0)\subset\R^2 $ with coordinates $x$ so that locally we have
$g_0= e^{2v_0}g_{\R^2}$ with $v_0\in C^{\infty}(\overline{B_1(0)})$. Setting
$v_n := u_n + v_0$, we obtain
\[
-\Delta v_n = (f_0(x)+\lambda_n)e^{2v_n} \;\;\; \mbox{on} \;\;\; B_1(0) .
\]
Split $v_n= v_n^{(0)}+v_n^{(1)}$, where $v_n^{(1)}\in H^1_0(B_1(0))$ solves the boundary value problem
\[
\left\{
\begin{array}{ll}
-\Delta v_n^{(1)} = (f_0(x)+\lambda_n) e^{2v_n} & \mbox{in} \;\; B_1(0), \\
v_n^{(1)} =  0 & \mbox{on} \;\; \partial B_1(0).
\end{array}
\right.
\]
and $v_n^{(0)}$ is harmonic. Hence it follows, uniformly in $n$,
\[
||\Delta v_n^{(1)}||_{L^{p}(B_1(0))}\leq||\Delta v_n^{(1)}||_{L^{\infty}(B_1(0))} \leq C
\] 
for any $p\geq 1$. Fixing $p>1$, from elliptic regularity theory we obtain that
$(v_n^{(1)})_n$ is bounded in $W^{2,p}(B_1(0))\hookrightarrow C^0({\overline{B_1(0)}})$. From the local upper bound on $\tilde{\Omega}$ for the sequence $(u_n)_n$ (and hence for $(v_n)_n$), we infer that for 
any $x\in \overline{B_1(0)}$,
\[
v_n^{(0)}(x) \leq ||v_n^{(1)}||_{L^{\infty}(B_1(0))} + C(\tilde{\Omega}) \leq C
\]
uniformly in $n$. Therefore, Harnack's inequality implies that 
\[
\sup_{B_{1/2}(0)} v_n^{(0)} \leq C_1\inf_{B_{1/2}(0)} v_n^{(0)} + C_2
\]
for suitable constants $C_1>0$ and $C_2\in\R$ depending on $B_{1/2}(0)$ but not on $n$.

We see that we have two mutually disjoint cases (up to subsequences):
\begin{enumerate}
  \item $\inf_{B_{1/2}(0)} v_n^{(0)}\to -\infty$, as $n\to +\infty$ 
  \item $\inf_{B_{1/2}(0)} v_n^{(0)} \geq -C$, uniformly in $n$.
\end{enumerate}
In the first case, it follows, recalling that $(v_n^{(1)})_n$ is bounded in $L^{\infty}(B_1(0))$, that
\[
v_n \to -\infty
\] 
uniformly in $\overline{B_{1/2}(0)}$.

In the second case, we deduce $C < v_n\big |_{\overline{B_{1/2}(0)}}$ uniformly in $n$.  

Since $\overline{\Omega} $ is connected, we conclude that either on $\overline{\Omega}$ the sequence of solutions $u_n$ goes uniformly to $-\infty$ or that there exists $C=C(\Omega)$ such that 
$u_n\big |_{\Omega} > C$ for any $n$. The Lemma is proved.
\end{proof}
%\begin{remark}\label{rmk: remark strumentale per la conlcusione del primo teorema}
%Assume that there exists $p_0\in M$ such that $f(p_0)=0$ and $\sup_{B_r(p_0)}u_n$ $< \infty$ uniformly in $n$ for some positive number $r$, that is, $p_0$ is not a blow-up point. Assume that the first case of Lemma \ref%{lemma: local lower bound for the sequence of solutions or convergence to minus inf} occurs; then, repeating %the argument present in the proof of the Lemma, it is easy to see that we have convergence to $-\infty$ also on the ball $B_r(p_0)$. 
%\end{remark}

\section{Proof of Theorem \ref{thm: second main result}}
In this section, we will analyze the asymptotic behaviour of the set of solutions to the
prescribed Gaussian curvature equation, when the parameter $\lambda\uparrow -\overline{f_0}=\lambda_{max}$. 
The main content of this section is the proof of Theorem \ref{thm: second main result}.

\begin{proposition}\label{prop: the energy vanishes}
Let $\beta_\lambda$ be defined by equation (\ref{eqn: definizione di beta lambda}). Then 
$\beta_\lambda \to 0$ as $\lambda \uparrow \lambda_{max}$.
\end{proposition}

In preparation for the proof of the Proposition, consider the Hilbert space $\Hone \times \R$ endowed with the natural scalar product and consider the set 
\begin{equation} \label{eqn: the hyperset C}
 \mathcal{C} := \left\{(u,\lambda) \in \Hone \times \R : \int_M u \,d\mu_{g_0} = 0 = \int_M f_\lambda e ^{2u}\,
d\mu_{g_0}\right\} .
\end{equation}
We claim that $\mathcal{C}$ is a $C^{\infty}$-Banach manifold. Indeed, we define 
$G:\Hone \times \R \rightarrow \R^2$ as:
\[
G(u,\lambda) := \left( \int_M u \,d\mu_{g_0} \, ;\int_M f_\lambda e ^{2u}\,d\mu_{g_0} \right) .
\]
Then
\[
G^{-1}((0,0)) = \mathcal{C}
\]
and $G\in C^{\infty}$ with first Frechet derivative
\[
 DG(u,\lambda)\left[v ,t\right] =  
 \left( \int_M v \,d\mu_{g_0} ; 2\int_M f_{\lambda} v e ^{2u}\,d\mu_{g_0} + t \int_M e^{2u} \, d\mu_{g_0}   \right)
\]
for any $(v,t)\in\Hone\times\R$.

For any $(u, \lambda) \in \mathcal{C}$, letting $DG(u,\lambda)$ act on $(1,0)$ and (0,1), we obtain
respectively the vectors $(1,0)$ and $ \left(0 , \int_M e^{2u}\, d\mu_{g_0} \right)$, which are clearly
a basis for $\R^2$. Moreover, the kernel of $DG(u,\lambda)$ splits $\Hone\times\R$. Thus, $\mathcal{C}$ is a smooth manifold of codimension equal to 2.

Define
\[
\mathcal{\tilde C}_\lambda := \mathcal{C} \cap \left\{(w,\mu)\in \Hone \times \R : \mu=\lambda \right\} 
\]
that is, the slice of $\mathcal{C}$ determined by the hyperplane in 
$\Hone \times \R$ of equation $\mu=\lambda$. We observe that this set is not empty for $\lambda\in(0,-\min_M f_0)$.
\begin{lemma}\label{lemma: riscrittura della slice}
There exist a function $s: \mathcal{C}_{\lambda_{max}} \to \R$ and a map 
$\Theta : \mathcal{C}_{\lambda_{max}}\times (0,-\min_M f_0) \to \Hone$ such that for any
$(u,\lambda)\in \mathcal{C}_{\lambda_{max}}\times (0,-\min_M f_0)$ we have
\[
u + s(u)(\lambda -\lambda_{max})(f_0 - \overline{f_0}) + \Theta(u, \lambda) \in \mathcal{C}_{\lambda} 
\]
and with the property that for any fixed $u\in \mathcal{C}_{\lambda_{max}}$
\[
|| \Theta(u,\lambda)||_{\Hone} = o(\lambda - \lambda_{max}) 
\]
as $\lambda\to\lambda_{max}$.
\end{lemma}
\begin{proof}
We take $u\in\mathcal{C}_{\lambda_{max}}$, $\lambda \in (0, -\min_M f_0)$ and consider the vector $\left(s(f_0-\overline{f_0}),1\right)\in\Hone \times \R $ where $s\in\R$. We want to find a suitable $s=s(u)$ 
such that the vector $\left(s(f_0-\overline{f_0}),1\right)$ belongs to the tangent space 
$T_{(u,\lambda_{max})} \mathcal{C}$. 

That amounts to impose
\[
DG(u,\lambda_{max})\left[s(f_0-\overline{f_0}),1\right] = (0,0)
\]
that is,
\[
\left(
\begin{array}{c}
 s  \int_M (f_0-\overline{f_0}) \, d\mu_{g_0}  \\
 2s  \int_M (f_0-\overline{f_0})^2 e^{2u} \, d\mu_{g_0} + \int_M e^{2u}\, d\mu_{g_0}
 \end{array}
\right) = 
\left(
\begin{array}{c}
 0 \\
 0
\end{array}
\right)
\]
Since $ \int_M (f_0-\overline{f_0}) \, d\mu_{g_0} =0$, we get from the second equation that 
\[
s(u)= - \frac{\int_M e^{2u}\, d\mu_{g_0}}{2\int_M (f_0-\overline{f_0})^2 e^{2u} \, d\mu_{g_0}} <0 .
\]
In view of the differentiable structure of $\mathcal{C}$, there exists 
$\Theta : \mathcal{C}_{\lambda_{max}}\times (0,-\min_M f_0) \to \Hone$ such that
\[
\left( u + s(u)(\lambda -\lambda_{max})(f_0 - \overline{f_0}) + \Theta(u, \lambda) ; \lambda \right)
\in \mathcal{\tilde C}_\lambda
\]
and $|| \Theta(u,\lambda)||_{\Hone} = o(\lambda - \lambda_{max})$ as $\lambda\to\lambda_{max}$.
The result follows.
\end{proof}

\begin{proof}[Proof of Proposition \ref{prop: the energy vanishes}]
We choose $u\equiv 0\in \mathcal{C}_{\lambda_{max}}$, $\lambda \in (0, -\min_M f_0)$ and compute $s$ and
$\Theta $ accordingly. Thus, 
$v_\lambda := s(0)(\lambda -\lambda_{max})(f_0 - \overline{f_0}) + \Theta(0, \lambda) 
\in \mathcal{C}_\lambda$; we evaluate its $\Hone$ norm
\begin{equation*}
\begin{split}
||s(0)(\lambda -\lambda_{max})(f_0 - \overline{f_0}) + \Theta(0, \lambda)||_{\Hone} \leq \\ \leq
|s(0)|\,|\lambda -\lambda_{max}|\,\,||f_0 - \overline{f_0} ||_{\Hone} + o(\lambda - \lambda_{max})  
\end{split}
\end{equation*}
and see that it goes to zero as $\lambda \to \lambda_{max}$.

Since for $\lambda<\lambda_{max}$ we have by definition $\beta_\lambda \leq E(v_\lambda)$, it follows $\beta_\lambda \to 0$ as $\lambda \uparrow \lambda_{max}$. 
\end{proof}

\begin{proof}[Proof of Theorem \ref{thm: second main result} (completed)]
Let $w_\lambda\in \mathcal{C}_\lambda$ be a minimizer
for $\lambda\in\Lambda$, as the one given in Section \ref{sec:Some notation and preliminary results}: 
then, since $\overline{w}_\lambda=0$ and $||\nabla w_\lambda||_{L^2(M)}^2=\beta_\lambda
\to 0$ when $\lambda \uparrow \lambda_{max}$, it follows by Poincar\'{e} -Wirtinger's inequality that
$w_\lambda \to 0$ in $\Hone$. 

Applying Moser-Trudinger's inequality, we also have
$e^{2w_\lambda}\to 1$ in $L^p(M)$ for any $p\in \left[1,\infty \right)$. Therefore, by H\"older's inequality, we obtain that for any $v\in\Hone$ 
\[
\int_M f_\lambda v e^{2w_\lambda} d\mu_{g_0} \to \int_M (f_0-\overline{f_0}) v d\mu_{g_0} 
\]
when $\lambda\uparrow \lambda_{max}$. We recall that, for any $\lambda\in\Lambda$, $w_\lambda$ solves
\[
\int_M (\nabla w_\lambda,\nabla v)_{g_0} d\mu_{g_0} = 
 \mu(\lambda) \int_M f_\lambda v e^{2w_\lambda} d\mu_{g_0}, \;\;\; v\in\Hone  
\]
where $\mu(\lambda)>0$ is a Lagrange multiplier. Choosing $v=f_0-\overline{f_0}$, we obtain for 
$\lambda \uparrow \lambda_{max}$
\[
0 = \lim_{\lambda \uparrow \lambda_{max}} \mu(\lambda) \;\int_M (f_0-\overline{f_0})^2 d\mu_{g_0}
\]
and therefore $\lim_{\lambda \uparrow \lambda_{max}} \mu(\lambda) =0 $.

Thus, using $L^p$-estimates, we obtain
\[
||w_\lambda||_{H^2(M)} \leq c\left(||\Delta w_\lambda ||_{L^2(M)} + ||w_\lambda||_{H^1(M)}    \right) .
\]
Since
\[
||\mu(\lambda)f_\lambda e^{2w_\lambda} ||_{L^2(M)} \leq 
\mu(\lambda) ||f_\lambda||_{\infty} \left[\int_M e^{4w_\lambda}d\mu_{g_0} \right]^{1/2}
\]
and $e^{4w_\lambda}\to 1$ in $L^1$ as $\lambda \uparrow \lambda_{max}$, it follows that 
$||\Delta w_\lambda ||_{L^2(M)}\to 0$ and hence $w_\lambda$ converges to zero in $H^2(M,g_0)$.
By Sobolev's embedding results, we also have for any $\alpha\in \left[0,1 \right)$
\[
w_\lambda \to 0 \;\;\mbox{in} \;\; C^{0,\alpha}(M)
\] 
when $\lambda \uparrow \lambda_{max}$.

Thus, using the bootstrap method and Schauder's estimates, we obtain $C^{2,\alpha}$ convergence as well.

Finally, we obtain that 
\[
u_\lambda := w_\lambda + 1/2 \log \mu(\lambda), 
\]
solution to equation (\ref{eqn: the Gauss curvature}), goes uniformly to $-\infty$ on $M$ when 
$\lambda \uparrow \lambda_{max}$ and therefore it can can not admit any convergent subsequence. 

This concludes the proof of Theorem \ref{thm: second main result}.
\end{proof}
\begin{remark}
Because of the conformal invariance of the Dirichlet energy and from convergence 
$||\nabla u_\lambda||_{L^2(M)}^2\to 0$ as $\lambda \uparrow -\overline{f_0}=\lambda_{max}$, it follows that no
``fine structure'' can appear in the ``limit'' geometry of the surfaces $\left(M,e^{2u_\lambda}g_0\right)$,
independently of how we blow up the scale. 
\end{remark}

\end{document}